\documentclass[12pt]{article}
\usepackage[utf8]{inputenc}
\usepackage{amsmath}
\usepackage{amssymb}
\usepackage[margin = 1.5 in, marginpar = 3 cm]{geometry}
\usepackage{graphicx}
\usepackage{amsthm}
\usepackage{mathtools}
\mathtoolsset{showonlyrefs}
\usepackage{hyperref}
\hypersetup{colorlinks,linkcolor={blue},urlcolor=cyan}

\usepackage{mleftright}

\usepackage{tikz}
\usepackage{theoremref}
\usepackage{wasysym}
\usepackage{harpoon}
\usepackage{float}
\usepackage{caption}
\usepackage{blindtext}
\usepackage{geometry}
\usepackage[title]{appendix}

\def\graphScale{0.29}

\title{Prime Graphs of Locally Torsion Ultagroups, Metanilpotent Groups, and $T$-Solvable Groups}
\author{Alexa Renner\footnote{arenner2@nd.edu; Department of Mathematics, Rose-Hulman Institute of Technology, 5500 Wabash Ave., Terre Haute, IN 47803, USA.}}

\DeclareMathOperator{\pgc}{\overline{\Gamma}}

\newtheorem{counter}{counter}[section]

\theoremstyle{definition}
\newtheorem{definition}[counter]{Definition}
\newtheorem*{definition*}{Definition}

\theoremstyle{plain}
\newtheorem{theorem}[counter]{Theorem}
\newtheorem{lemma}[counter]{Lemma}
\newtheorem{corollary}[counter]{Corollary}
\newtheorem{proposition}[counter]{Proposition}
\newtheorem{fact}[counter]{Fact}

\theoremstyle{remark}
\newtheorem{remark}[counter]{Remark}

\newcommand{\notIsolatedVertexWhite}{
\begin{tikzpicture}[scale=\graphScale]
    \draw (0, 0) -- (1, 0);
    \draw (0, 0) -- (0, 1);
    \draw (1, 0) -- (0, 1);
    \draw[fill=black] (0, 0) circle (0.125);
    \draw[fill=black] (0, 1) circle (0.125);
    \draw[fill=white] (1, 0) circle (0.125);
    \draw[fill=black] (1, 1) circle (0.125);
\end{tikzpicture}
}
\newcommand{\notIsolatedVertexWhiteGrey}{
\begin{tikzpicture}[scale=\graphScale]
    \draw (0, 0) -- (1, 0);
    \draw (0, 0) -- (0, 1);
    \draw (1, 0) -- (0, 1);
    \draw[fill=black] (0, 0) circle (0.125);
    \draw[fill=black] (0, 1) circle (0.125);
    \draw[fill=white] (1, 0) circle (0.125);
    \draw[fill=black] (1, 1) circle (0.125);
\end{tikzpicture}
}

\newcommand{\hamSandwichWhite}{
\begin{tikzpicture}[scale=\graphScale]
    \draw (0, 0) -- (1, 0);
    \draw (0, 0) -- (0, 1);
    \draw (1, 0) -- (1, 1);
    \draw (1, 0) -- (0, 1);
    \draw (0, 1) -- (1, 1);
    \draw[fill=black] (0, 0) circle (0.125);
    \draw[fill=black] (0, 1) circle (0.125);
    \draw[fill=white] (1, 0) circle (0.125);
    \draw[fill=black] (1, 1) circle (0.125);
\end{tikzpicture}
}
\newcommand{\SuzukiThirtyTwoElusiveGraphWhiteBadb}{
\begin{tikzpicture}[scale=\graphScale]
    \draw (0, 0) -- (1, 0);
    \draw (0, 0) -- (0, 1);
    \draw (0, 0) -- (1, 1);
    \draw (1, 0) -- (1, 1);
    \draw[fill=black] (0, 0) circle (0.125);
    \draw[fill=black] (0, 1) circle (0.125);
    \draw[fill=white] (1, 0) circle (0.125);
    \draw[fill=black] (1, 1) circle (0.125);
\end{tikzpicture}
}
\newcommand{\szunrecognizei}{
\begin{tikzpicture}[scale=\graphScale]
    \draw (0, 0) -- (1, 0);
    \draw (0, 1) -- (1,1);
    \draw (1, 0) -- (1, 1);
    \draw (1, 0) -- (0, 1);
    \draw[fill=black] (0, 0) circle (0.125);
    \draw[fill=black] (0, 1) circle (0.125);
    \draw[fill=white] (1, 0) circle (0.125);
    \draw[fill=black] (1, 1) circle (0.125);
\end{tikzpicture}
}

\newcommand{\psltwosixteenimpossiblei}{
\begin{tikzpicture}[scale=\graphScale]
    \draw (1, 1) -- (0, 1);
    \draw (1, 1) -- (0, 0);
    \draw (0, 1) -- (0, 0);
    \draw[fill=black] (0, 0) circle (0.125);
    \draw[fill=black] (0, 1) circle (0.125);
    \draw[fill=white] (1, 0) circle (0.125);
    \draw[fill=black] (1, 1) circle (0.125);
\end{tikzpicture}
}
\newcommand{\psltwosixteenimpossibleii}{
\begin{tikzpicture}[scale=\graphScale]
    \draw (1, 1) -- (0, 1);
    \draw (1, 1) -- (0, 0);
    \draw (1, 1) -- (1, 0);
    \draw (0, 1) -- (0, 0);
    \draw[fill=black] (0, 0) circle (0.125);
    \draw[fill=black] (0, 1) circle (0.125);
    \draw[fill=white] (1, 0) circle (0.125);
    \draw[fill=black] (1, 1) circle (0.125);
\end{tikzpicture}
}
\newcommand{\psltwosixteenimpossibleiii}{
\begin{tikzpicture}[scale=\graphScale]
    \draw (1, 1) -- (0, 1);
    \draw (1, 1) -- (0, 0);
    \draw (1, 1) -- (1, 0);
    \draw (0, 0) -- (1, 0);
    \draw (0, 1) -- (0, 0);
    \draw[fill=black] (0, 0) circle (0.125);
    \draw[fill=black] (0, 1) circle (0.125);
    \draw[fill=white] (1, 0) circle (0.125);
    \draw[fill=black] (1, 1) circle (0.125);
\end{tikzpicture}
}

\begin{document}
\maketitle
\begin{abstract}
    Given a group $G$,   
    the prime graph of $G$, denoted $\Gamma(G)$, is the graph with vertex set the set of primes $p$ such that there exists an element of order $p$ in $G$, and such that there is an edge between primes $p$ and $q$ if and only if there exists an element of order $pq$ in $G$. Let $T$ be nonabelian finite simple group. A $T$-solvable group $G$ is a group such that there exists a finite subnormal series of $G$ such that each factor is either abelian or isomorphic to $T$. In this paper, we continue the classification of the prime graphs of $T$-solvable locally torsion groups and classify prime graphs of locally torsion metanilpotent groups using tools from model theory. We also classify the prime graphs of ultragroups with certain first-order properties and prove results on the number and type of certain classes of groups that can realize a given graph as a prime graph.
\end{abstract}
\hfill\break

\hfill\break
\section*{Introduction}
The prime graph of a (possibly infinite) group is the graph with vertex set $\pi(G) = \{p:p\in\mathbb{P} \text{ and there exists an element }g\in G\text{ with }o(g) = p\}$ and an edge between $p$ and $q$ if and only if there exists an element of order $pq$ in $G$. The prime graphs of finite groups have been studied extensively, in papers such as \cite{baechle}, \cite{angel}, and \cite{Florez}, although there still remain some notable open questions (see \cite{maslovanotebook}). A major open question is whether every finite graph can be realized as the prime graph of a finite group.

The prime graphs of infinite groups have not been as thoroughly studied. In the introduction to \cite{angel}, the authors mention that any countable graph can be realized as a free product of the appropriate cyclic groups. Even though the question of whether every countable graph can be realized as the prime graph of a group has been solved, there are still several other interesting questions that arise, such as classifying the prime graphs of certain classes of groups. 

To be more specific, we need to introduce a few definitions:
\begin{definition}
    A locally torsion group is a group $G$ such that for any finite $\pi\subseteq\pi(G)$, there exists a torsion $H\leq G$ such that $\pgc(G)[\pi] = \pgc(H)$.
\end{definition}
\begin{definition}
    We say a group $G$ has the finite Frobenius property if, for every $p\in\pi(G)$, there are only finitely many $q\in\pi(G)$ such that no element of order $pq$ exists in $G$.
\end{definition}
In \cite{2015REU}, the authors proved the following result:
\begin{theorem}
    \cite[Theorem 2.8]{2015REU}: Let $\Xi$ be a finite graph. $\Xi$ is isomorphic to the prime graph complement of a finite solvable group if and only if $\Xi$ is triangle-free and 3-colorable.
\end{theorem}
This result was generalized to the following in \cite{infinitePGCs}:
\begin{theorem}
    \cite[Theorem 3.11]{infinitePGCs}: Let $\Xi$ be a graph such that each vertex has finite degree. Then, $\Xi$ is isomorphic to the prime graph complement of a solvable torsion group with the finite Frobenius property if and only if $\Xi$ is triangle-free and 3-colorable.
\end{theorem}

The classification of the prime graphs of finite solvable groups inspired the following definition:
\begin{definition}
    Let $\mathcal{T}$ be a collection of nonabelian finite simple groups. A finite group $G$ is $\mathcal{T}$-solvable if there exists a composition series of $G$ such that each factor is isomorphic to a cyclic group or an element of $\mathcal{T}$. If $\mathcal{T} = \{T\}$, we will refer to $\mathcal{T}$-solvable as $T$-solvable.
\end{definition}
There has been much work towards the classification of the prime graphs of finite $T$-solvable. The authors of \cite{a5paper} and \cite{2022REU} completely classified the prime graphs of $\mathcal{T}$-solvable groups where each $T\in\mathcal{T}$ has $|\pi(T)| = 3$. The case where $|\pi(T)| = 4$ is still open, but \cite{2023REU}, \cite{Suz}, \cite{psl}, and \cite{alland} have classified the prime graph complements of certain $T$-solvable groups where $\pi(T) = 4$. This motivated the following definition:
\begin{definition}
    Let $\mathcal{T}$ be a collection of nonabelian finite simple groups. A group $G$ is $\mathcal{T}$-solvable if there exists a subnormal series of $G$ such that each factor is isomorphic to an abelian group or an element of $\mathcal{T}$. If $\mathcal{T} = \{T\}$, we will refer to $\mathcal{T}$-solvable as $T$-solvable.
\end{definition}
Notice that for finite groups, the definitions of $\mathcal{T}$-solvable coincide. Many of the classification theorems for finite $T$-solvable groups were generalized to torsion and locally torsion $T$-solvable groups with the finite Frobenius property in \cite{infinitePGCs}.

In this paper, we use tools from model theory to extend the classifications of finite $T$-solvable groups to locally torsion groups that were unable to be extended using the algebraic methods presented in \cite{infinitePGCs}; and to classify the prime graphs of other groups of interest. Additionally, we prove some results on the number and type of groups realizing certain prime graphs using model-theoretic tools. As is consistent with \cite{Suz}, \cite{psl}, and \cite{infinitePGCs}, we state our classification results in terms of the prime graph complements rather than the prime graphs. Throughout the paper, we assume choice.

In section \ref{ultra}, we classify the prime graph complements of ultragroups that satisfy certain first-order sentences and pseudofinite groups. In section \ref{metanp}, we classify the prime graph complements of locally torsion metanilpotent groups. In section \ref{tsolv}, we prove some lemmas that will be helpful in constructing locally torsion $T$-solvable groups realizing certain prime graph complements. Section \ref{tsolvconstruction} uses these results to classify the prime graphs of $T$-solvable locally torsion groups with the finite Frobenius property for certain $T$. Section \ref{sec7} proves results on the number of certain types of nonisomorphic groups that have a certain graph as their prime graph, and also provides a result on when an infinite graph is realized by a torsion group. Finally, section \ref{outlook} provides some ideas for future work.
\subsection*{Notation and Conventions}
We adopt the convention $0\in\mathbb{N}$.
\begin{itemize}
    \item Let $\Xi$ be a graph. We denote the vertex set of $\Xi$ by $V(\Xi)$ and the edge set of $\Xi$ by $E(\Xi)$. 
    \item Let $\Xi$ be a graph and let $X\subseteq V(\Xi)$. We denote the subgraph induced by $X$ as $\Xi[X]$, and we denote the neighborhood of $X$ by $N[X]$. We will denote  the open neighborhood of $X$, the set $N[X]\setminus X$, by $N(X)$.
    \item $\mathbb{N}_{\geq k} = \{i:i\in\mathbb{N},i\geq k\}$.
    \item Let $G$ be a group. We denote the order of an element $g\in G$ by $o(g)$.
    \item Let $\mathcal{L}_{gp} = \{^{-1}, 1,\cdot\}$. We will refer to $\mathcal{L}_{gp}$ as the language of groups.
    \item $\mathbb{P}$ denotes the set of primes of $\mathbb{N}$.
    \item Let $G$ be a group. Define $\pi(G) = \{p:p\in\mathbb{P} \text{ and there exists an element }g\in G\text{ with }o(g) = p\}$.
    \item Let $G$ be a group. Define the spectrum of $G$, $\omega(G)$, to be $\{o(g):g\in G\}$.
\end{itemize}
\section{Ultragroups}\label{ultra}
\begin{lemma}\thlabel{uf}
    Let $\mathcal{L}$ be a language. Let $\mathcal{V}$ be a nonprincipal ultrafilter on an infinite set $I$, $\mathcal{M} = \prod_\mathcal{V}\mathcal{M}_\alpha$ an ultraproduct. Let $\sigma_0,\sigma_1,\sigma_2,\cdots$ be $\mathcal{L}$-sentences true in $\mathcal{M}$. There exist $\alpha_0,\alpha_1,\alpha_2,\cdots\in I$ such that, for each $n\in\mathbb{N}$, $\mathcal{M}_{\alpha_n}\vDash\sigma_i$ for all $i\leq n$.
\end{lemma}
\begin{proof}
    For each $i\in\mathbb{N}$, let $A_i = \{\alpha:\mathcal{M}_\alpha\vDash\sigma_i\}$, and define $B_i = \bigcap_{j\leq i}A_i$. Each $B_i\in\mathcal{V}$, so each $B_i\neq\varnothing$. For each $i\in\mathbb{N}$, let $\alpha_i\in B_i$. Then, $\mathcal{M}_{\alpha_i}\vDash\sigma_j$ for all $j\leq i$.
\end{proof}
\begin{theorem}\thlabel{ultraclas}
   Let $\Xi$ be a graph. Let $\mathcal{L}$ be a language which includes $\cdot$, and let $\varphi$ be an $\mathcal{L}$-sentence. $\Xi$ is isomorphic to the prime graph complement of a nonprincipal ultragroup over $\mathcal{L}$ in which $\varphi$ holds, where $\cdot$ is interpreted as the group operation, if and only if there exists a labeling of $\Xi_L$ of $\Xi$, a sequence of subgraphs $\{\Lambda_i\}_{i\in\mathbb{N}}$ of $\Xi_L$ such that $\bigcup_{i\in\mathbb{N}}\Lambda_i = \Xi_L$, $\Lambda_i = \Lambda_{i+1}[V(\Lambda_i)]$ for all $i\in\mathbb{N}$, and a sequence of groups $\{G_i\}_{i\in\mathbb{N}}$ such that:
   \begin{itemize}
       \item For all $i\in\mathbb{N}, \Lambda_i$ is an induced subgraph of $ \pgc(G_i)$.
       \item For all $i\in\mathbb{N}, G_i\vDash\varphi$.
       \item For all $p\in\mathbb{P}$, the set $V_p = \{i:p\in V(\pgc(G_i))\setminus V(\Xi_L)\}$ is finite.
   \end{itemize}
\end{theorem}
\begin{proof}
    Let $G = \prod_\mathcal{U}G_\alpha$. Let $\{\Lambda_i\}_{i\in\mathbb{N}}$ be any nested sequence of subgraphs such that $\bigcup\Lambda_i = \Xi_L$. Enumerate $\mathbb{P}\setminus \pi(G)$ as $p_0,p_1,\cdots$. For each $i\in\mathbb{N}$, let $\sigma_i$ encode $\Lambda_i$ and let $\tau_i$ be $\neg\exists x(x^{p_i}=1\land(\bigwedge_{j=1}^{p_i-1}x^j\neq 1))$. Then, by \thref{uf}, there exists a sequence $\{\alpha_n\}_{n\in\mathbb{N}}$ such that $G_{\alpha_n}\vDash\varphi\land \sigma_i\land(\bigwedge_{j=0}^i\tau_j)$ if $i\leq n$. 

    Fix a nonprincipal ultrafilter $\mathcal{U}$ on $\mathbb{N}$, and let $G =\prod_\mathcal{U}G_i$. Enumerate $\mathbb{P}\setminus V(\Xi_L)$ as $p_0,p_1,\cdots$. For each $i\in\mathbb{N}$, let $\sigma_i$ encode $\Lambda_i$ and let $\tau_i$ be $\neg\exists x(x^{p_i}=1\land(\bigwedge_{j=1}^{p_i-1}x^j\neq 1))$. Then, because $G_n\vDash\varphi\land \sigma_i\land(\bigwedge_{j\leq i}\tau_j)$ for all $i\leq n,$ $G\vDash \sigma_i$ and $G\vDash\tau_i$ for all $i\in\mathbb{N}$, and $G\vDash\varphi$. By choice of $\tau_i,\sigma_i,\pgc(G) = \Xi_L$.
\end{proof}
\begin{definition}
    A pseudofinite group is an infinite group which satisfies every first order sentence of $\mathcal{L}_{gp}$ that is true in every finite group.
\end{definition}
We now define a new concept which will be useful in this section and \ref{sec7}:
\begin{definition}
    Let $\mathcal{L}$ be a language and let $\mathcal{G},\mathcal{H}$ be $\mathcal{L}$-structures. $\mathcal{G}$ is elementarily equivalent to $\mathcal{G}$ (denoted $\mathcal{G}\equiv\mathcal{H}$) if, for any $\mathcal{L}$-sentence $\varphi, \mathcal{G}\vDash\varphi$ if and only if $\mathcal{H}\vDash\varphi$.
\end{definition}
\begin{lemma}\thlabel{elmequiv}
    Let $T$ be the theory of groups in $\mathcal{L}$ for some language $\mathcal{L}$ containing a binary function symbol and $\mathcal{G},\mathcal{H}$ be $\mathcal{L}$-structures satisfying $T$. Then, if $\mathcal{H}\equiv\mathcal{G}, \pgc(\mathcal{H}) = \pgc(\mathcal{G})$. In particular, if $\mathcal{H}\preceq\mathcal{G}, \pgc(\mathcal{H}) = \pgc(\mathcal{G})$. 
\end{lemma}
\begin{proof}
     Let $n\in\mathbb{N}$. Notice that $\mathcal{G}\vDash\exists x(x^n = 1\land (\bigwedge_{j=1}^{n-1}x^j\neq 1))$ if and only if $\mathcal{H}\vDash\exists x(x^n = 1\land (\bigwedge_{j=1}^{n-1}x^j\neq 1))$. Thus $\mathcal{G}$ contains an element of order $n$ if and only if $\mathcal{H}$ contains an element of order $n$, so $\pgc(\mathcal{G}) = \pgc(\mathcal{H})$. By definition of $\preceq$, if $\mathcal{H}\preceq\mathcal{G}$, then $\mathcal{H}\equiv\mathcal{G}$, so the claim follows.
\end{proof}
\begin{corollary}
   Let $\Xi$ be a graph. $\Xi$ is isomorphic to the prime graph complement of a pseudofinite group if and only if there exists a labeling $\Xi_L$ of $\Xi$, a sequence of induced subgraphs $\{\Lambda_i\}_{i\in\mathbb{N}}$ of $\Xi_L$ such that $\bigcup_{i\in\mathbb{N}}\Lambda_i = \Xi_L$ and for all $i\in\mathbb{N}, \Lambda_i$ is an induced subgraph of $\Lambda_{i+1}$, and a sequence of distinct finite groups $\{G_i\}_{i\in\mathbb{N}}$ such that:
   \begin{itemize}
       \item For all $i\in\mathbb{N}, \Lambda_i$ is an induced subgraph of $\pgc(G_i)$.
       \item For all $p\in\mathbb{P}$, the set $V_p = \{i:p\in V(\pgc(G_i))\setminus V(\Xi_L)\}$ is finite.
   \end{itemize}
\end{corollary}
\begin{proof}
    This follows from \thref{elmequiv} and the proof of \thref{ultraclas}.
\end{proof}

One may wonder whether the prime graph of an ultraproduct of groups $\prod_\mathcal{U}G_i$ (over a nonprincipal ultrafilter $\mathcal{U}$) is always isomorphic to the ultraproduct of the unlabeled prime graphs of the $G_i$ over some ultrafilter on the same index set. This is false in general. As a very simple counterexample, let $\mathcal{U}$ be a nonprincipal ultrafilter over $\mathbb{N}$, enumerate $\mathbb{P}$ as $p_0,p_1,p_2,\cdots$, and consider $\prod_{\mathcal{U}} C_{p_i}$. We claim that $\pgc\left(\prod_{\mathcal{U}} C_{p_i}\right)$ is the empty graph. For contradiction, suppose $\pgc\left(\prod_{\mathcal{U}} C_{p_i}\right)$ had a vertex $p$. Let $\varphi$ be $\exists x(x^p = 1\land(\bigwedge_{j=1}^{p-1}x^j\neq 1))$. We have that $\prod_\mathcal{U}C_{p_i}\vDash\varphi$, so by Łos's theorem, there exists a set $A\in\mathcal{U}$ such that for all $i\in A, G_i\vDash\varphi$. Because $C_{p_i}\vDash\varphi$ if and only if $p = p_i, A = \{p\}$, so $\mathcal{U}$ is principal, a contradiction. Thus, $\pgc\left(\prod_\mathcal{U}G_i\right)$ is the empty graph.

Let $\mathcal{V}$ be a nonprincipal ultrafilter on $\mathbb{N}$. For each $i, $ let $\Lambda_i$ be the unlabeled graph corresponding to $\pgc(C_{p_i})$. For each $i\in\mathbb{N}$, this graph is the graph consisting of a single vertex. Then, let $\psi$ be $\exists v_1\forall v_2(v_1 = v_2)$, and let $A$ be the set of indices in $\mathbb{N}$ on which this holds. Notice that $\Lambda_i\vDash\psi$ for all $i\in\mathbb{N}$, so $A = \mathbb{N}$ which is in $\mathcal{V}$, so $\prod_\mathcal{V}\Lambda_i\vDash\psi$ by Łos's theorem. Thus, $\prod_\mathcal{V}\Lambda_i\vDash\psi$ is the graph on one vertex, and thus $\pgc\left(\prod_{\mathcal{U}} C_{p_i}\right)$ is not isomorphic to $\prod_\mathcal{V}\Lambda_i$ for any nonprincipal ultrafilter $\mathcal{V}$ on $\mathbb{N}$.
\section{Metanilpotent Groups}\label{metanp}
\begin{lemma}\thlabel{2colors}
    Let $G$ be a finite group which is one of the following:
    \begin{itemize}
        \item Metacyclic
        \item Metabelian
        \item Metanilpotent
        \item Supersolvable
    \end{itemize}
     Then $\pgc(G)$ is 2-colorable.
\end{lemma}
\begin{proof}
This follows from the remark after \cite[Theorem 3.1]{Florez}.

\end{proof}
\begin{lemma}\thlabel{cyclicsmall}
    Let $G$ be an infinite cyclic group. Then, $\pi(G) =\varnothing$.
\end{lemma}
\begin{lemma}\thlabel{supersolv}
    Let $G$ be a supersolvable group. Then, $\pi(G)$ is finite.
\end{lemma}
\begin{proof}
    Let $1 = N_0\unlhd\cdots\unlhd N_k = G$ be a subnormal series of $G$ such that, for all $0\leq i\leq k-1, N_{i+1}/N_i$ is cyclic. For contradiction, suppose $\pi(G)$ were infinite. Notice that $\bigcup_{0\leq i\leq k-1}\pi(N_{i+1}/N_i)$ is finite by \thref{cyclicsmall}, so there exists $p\in\pi(G)\setminus(\bigcup_{0\leq i\leq k-1}\pi(N_{i+1}/N_i))$. Let $g\in G$ have order $p$, and let $i\in\mathbb{N}$ be the $i\in\mathbb{N}$ such that $g\not\in N_i, g\in N_{i+1}$. Then, $o(gN_i) = p\in\pi(N_{i+1}/N_i)$, a contradiction.
\end{proof}
\begin{theorem}\thlabel{supersolvthm}
    Let $\Xi$ be a graph. $\Xi$ is the prime graph complement of a locally torsion supersolvable group if and only if $\Xi$ is finite and 2-colorable.
\end{theorem}
\begin{proof}
    Let $G$ be a supersolvable locally torsion group. By \thref{supersolv}, $\pi(G) = V(\pgc(G))$ is finite, so there exists a torsion $H\leq G$ such that $H[\pi(G)] = \pgc(G)$. Because $H\leq G$, $H$ is supersolvable. Let $1 = N_0\unlhd\cdots\unlhd N_k = G$ be a subnormal series of $H$ such that, for all $0\leq i\leq k-1, N_{i+1}/N_i$ is cyclic. By \thref{cyclicsmall}, $N_{i+1}/N_i$ is finite for all $0\leq i\leq k-1$, so $H$ is finite. Applying the remark after \cite[Theorem 3.1]{Florez}, $\pgc(H) =\pgc(G)$ is 2-colorable.

    Now suppose $\Xi$ is finite and 2-colorable. By the remark after \cite[Theorem 3.1]{Florez}, there exists a finite supersolvable group $G$ realizing $\Xi$.
\end{proof}
\begin{corollary}
    Let $\Xi$ be an unlabeled graph. $\Xi$ is the prime graph complement of a locally torsion metacyclic group if and only if $\Xi$ is finite and 2-colorable.
\end{corollary}
\begin{proof}
    By \thref{supersolvthm} and the remark after \cite[Theorem 3.1]{Florez}.
\end{proof}
\begin{definition}
    Let $G$ be a finite group and $\pi$ a set of primes. A \textbf{Hall }$\pi$\textbf{-subgroup} of $G$ is a subgroup $H\leq G$ such that for any $p\in \pi, p\nmid |G/H|$ and $\pi(H)\subseteq\pi$.
\end{definition}
Recall that Hall $\pi$-subgroups are guaranteed to exist in finite solvable groups.
\begin{theorem}\thlabel{metabelian}
Let $\Xi$ be a countable graph. $\Xi$ is isomorphic to the prime graph complement of a locally torsion metanilpotent group if and only if $\Xi$ is 2-colorable.
\end{theorem}
\begin{proof}
First consider the forwards direction. Let $G$ be a locally torsion metanilpotent group, define $\Xi$ to be the unlabeled prime graph complement of $G$, and let $\Delta$ be a finite induced subgraph of $\pgc(G)$. Let $H\leq G$ be a torsion subgroup of $G$ such that $\pgc(H)[V(\Delta)] = \pgc(G)[V(\Delta)]$. For each $p\in V(\Delta)$, let $g_p\in H$ be an element of order $p$, and for each $p-q\not\in \Delta$, let $g_{pq}\in H$ be an element of order $pq$. By \cite[Lemma 2.4]{infinitePGCs},  $K = \langle g_p, g_{pq}:p\in V(\Delta), p-q\not\in\Delta\rangle$ is a finite solvable group, and by the fact that $K\leq G$, $K$ is metanilpotent. Let $M$ be a Hall $V(\Delta)$-subgroup of $K$. Notice that $\pgc(M) =\pgc(H)[V(\Delta)]= \Delta$. Because $M$ is metanilpotent, by \thref{2colors}, $\pgc(M)$ is 2-colorable. Recall that $\Delta$ was an arbitrary finite induced subgraph of $\pgc(G)$, so by \cite[Lemma 2.1]{infinitePGCs}, $\pgc(G)$ is 2-colorable.

Consider the backwards direction. Color $\Xi$ with colors $\mathcal{O}$ and $\mathcal{I}$. Label $\Xi$ with $\mathbb{P}$ in such that any two distinct vertices have distinct primes as labels. Call the labeled graph $\Xi_L$. Let $A$ be the set of prime and vertex pairs $\langle p, v\rangle$ such that the corresponding color of the vertex is $\mathcal{O}$, and $B$ be the set of prime and vertex pairs such that the corresponding color of the vertex is $\mathcal{I}$. Enumerate $A$ as $\langle v_0,p_0\rangle,\langle v_1,p_1\rangle,\cdots$ and $B$ as $\langle u_0,q_0\rangle,\langle u_1,q_1\rangle,\cdots$. We define a sequence $\{D_s\}_{s\in\mathbb{N}}$ of finite sets as follows:

Define $D_0  =\varnothing$.

If $s = 2i:$ Define $D_{s+1} = D_s\cup\{\langle v_{i},p_{i}\rangle\}$.

If $s = 2i + 1:$ Define $D_{s+1} = D_s\cup\{\langle u_{i},q_{i}\rangle\}$.

Notice that for any vertex $v\in V(\Xi)$ labeled with $p$ in $\Xi_L$, there exists an $s\in\mathbb{N}$ such that $\langle v,p\rangle\in D_s$. We now define a sequence of groups $\{G_i\}_{i\in\mathbb{N}}$ as follows:

Enumerate $A\cap D_i$ as $\langle v_0,p_0\rangle,\cdots,\langle v_{n_i},p_{n_i}\rangle$ and $B\cap D_i$ as $\langle u_0,q_0\rangle,\cdots,\langle u_{m_i},q_{m_i}\rangle$. For each $q_j, 0\leq j\leq m_i$, define $\pi_{i,j} = \{p_k:p_k-q_j\in\Xi_L, 0\leq k\leq n_i\}$.

By Fermat's little theorem, $p_k\mid q_j^{p_k-1}$ for all $0\leq k\leq n_i$. Notice that this implies $1\equiv q_j^{p_k-i}\mod p_k$ for all $0\leq k\leq n_i$, so $1\equiv q_j^{\prod_{k=1}^{n_i}(p_k-1)}\mod p_t$ for $0\leq t\leq n_i$. Let $\mathbb{F}_{q_j,i}$ be the field of order $q_j^{\prod_{k=1}^{n_i}(p_k-1)}$. Notice that the multiplicative group $\mathbb{F}_{q_j,i}^\times$ embeds into the automorphism group of the additive group $\mathbb{F}_{q_j,i}^+$. Recall that $|\mathbb{F}_{q_j,i}^\times| = q_j^{\prod_{k=1}^{n_i}(p_k-1)}-1$, so $p_t\mid |\mathbb{F}_{q_j,i}^\times|$. Thus, for each $k\in \pi_{i,j}$, there exists a cyclic subgroup $C_{i,p_k}\leq\mathbb{F}_{q_j,i}^\times\leq\operatorname{Aut}(\mathbb{F}_{q_j,i}^+)$ of order $p_k$. Define $Q_{i,q_j} = \mathbb{F}_{q_j,i}^+$.

Define $P_i = C_{p_0}\times\cdots\times C_{p_{n_i}}$ and define $Q_i = Q_{i,q_0}\times\cdots\times Q_{i,q_{m_i}}$. Notice that to define a homomorphism $\varphi_i:P_i\to\operatorname{Aut}(Q_i)$ it suffices to define homomorphisms $\varphi_{i,j}:P_i\to \operatorname{Aut}(Q_{i,q_j})$ for $0\leq j\leq m_i$. To do this, we first define homomorphisms $\gamma_{i,q_j,k}:C_{p_k}\to \operatorname{Aut}(Q_{i,q_j})$ for $0\leq k\leq n_i$. For each $0\leq k\leq n_i$, if $k\in \pi_{i,j}$ let $\gamma_{i,q_j,k}$ be an isomorphism of $C_{p_k}$ and $C_{i,p_k}$ and if $k\not\in \pi_{i,j}$, let $\gamma_{i,q_j,k}$ be the trivial homomorphism. Now, for each $0\leq j\leq m_i$, define $\varphi_{i,q_j} = \gamma_{i,q_j,n_i}\circ\cdots\circ\gamma_{i,q_j,0}$. Define $\varphi_i(x_0,\cdots,x_{m_i}) = (\varphi_{i,q_0}(x_0),\cdots,\varphi_{i,q_{m_i}}(x_{m_i}))$ and notice that because $\varphi_{i,q_j}$ is a homomorphism for $0\leq j\leq m_i$, $\varphi_i$ is a homomorphism. Let $G_i = P_i\ltimes_{\varphi_i}Q_i$.

Let $I\subseteq\mathcal{P}(\mathbb{P})$ be the set of all finite subsets of $\mathbb{P}$. Define $\mathcal{L} = \mathcal{L}_{gp}\cup\{N\}\cup\{H_\pi:\pi\in I\}$ where $N$ and $H_\pi$ are unary relation symbols for $\pi\in I$. We will view each $G_i$ as a $\mathcal{L}$-structure in the following way: Assign the usual interpretations to the symbols of $\mathcal{L}_{gp}$, interpret $N$ such that $N^{G_i} = Q_i$, and for each $\pi\in I$, interpret $H_\pi$ as the following Hall $\pi$-subgroup of $G_i$:
\begin{align*}
    \{((x_0,\cdots,x_{m_i}),(y_0,\cdots,y_{n_i}))\in G_i:x_j = 1\text{ if }q_j\not\in D_i\text{ and }x_k = 1\text{ if }p_k\not\in D_i\}.
\end{align*}
Let $\psi(x,y)$ be the formula $\exists a[N(a)\land x\cdot a = y]$. Define a set of $\mathcal{L}$-sentences $\Phi$ as follows:
\begin{align*}
    \forall x,y,z[(x\cdot y)\cdot z = x\cdot (y\cdot z)]\\
    \forall x[x\cdot 1 = 1\cdot x = x]\\
    \forall x[x\cdot x^{-1} = x^{-1}\cdot x = 1]\\
    \forall x,y[N(x)\land N(y)\to N(x\cdot y)]\\
    \forall x[N(x)\to N(x^{-1})]\\
    \forall x,y[N(x)\land N(y)\to x\cdot y = y\cdot x ]\\
    \forall x,y[N(x)\to N(y^{-1}\cdot x\cdot y)]\\
    \forall x, y[\psi(x\cdot y,y\cdot x)]\\
    \forall x, y[H_\pi(x)\land H_\pi(y)\to H_\pi(x\cdot y)]\text{ for all }\pi\in I\\
    \forall x[H_\pi(x)\to H_\pi(x^{-1})]\text{ for all }\pi\in I\\
    \forall x[H_\pi(x)\to x^{\prod_{r\in\pi}r}=1]\text{ for all }\pi\in I.
\end{align*}
Notice that the first eight axioms are satisfied in $G_i$ for all $i\in\mathbb{N}$ by choice of interpretation of $N$, as are the first two schema involving the $H_\pi$.

\hfill\break
\textbf{Claim 1:} For all $\pi\in I, i\in\mathbb{N}, G_i\vDash\forall x[H_\pi(x)\to x^{\prod_{r\in\pi}r}=1]$.

\hfill\break
\textbf{Proof:} Let $((x_0,\cdots,x_{m_i}),(y_0,\cdots,y_{n_i}))\in H_\pi^{G_i}$. By abuse of notation, we will denote the group operation in $Q_{i,q_j}$ as $+$ for $0\leq j\leq m_i$. Let $0\leq j\leq m_i$, and let $1\leq k\leq \prod_{p\in\pi\cap A}p$. Define $a_{j,k} = \varphi_{i,q_j}(y_0^{k-1},\cdots,y_{n_i}^{k-1})$. Notice that 
\begin{align*}
    ((x_0,\cdots,x_{m_i}),(y_0,\cdots,y_{n_i}))^{\prod_{p\in\pi\cap A}p} = (((\sum_{k=1}^{\prod_{p\in\pi\cap A}p}a_{0,k})x_0,\cdots,(\sum_{k=1}^{\prod_{p\in\pi\cap A}p}a_{m_i,k})x_{m_i}),(1,\cdots,1))
\end{align*}
Then, 
\begin{align*}
    ((x_0,\cdots,x_{m_i}),(y_0,\cdots,y_{n_i}))^{\prod_{p\in\pi}p} = \left(((x_0,\cdots,x_{m_i}),(y_0,\cdots,y_{n_i}))^{\prod_{p\in\pi\cap A}p}\right)^{\prod_{q\in\pi\cap B}q} \\
    = (((\sum_{k=1}^{\prod_{p\in\pi\cap A}p}a_{0,k})x_0,\cdots,(\sum_{k=1}^{\prod_{p\in\pi\cap A}p}a_{m_i,k})x_{m_i}),(1,\cdots,1))^{\prod_{q\in\pi\cap B}q} = ((0,\cdots,0),(1,\cdots,1))
\end{align*}
and the claim is proved.
\begin{flushright}
$\dashv_{\text{Claim}}$
\end{flushright}
By claim 1, $G_i\vDash\Phi$ for all $i\in\mathbb{N}$.

\hfill\break
\textbf{Claim 2:} For all $i\in\mathbb{N}$, $\pgc(G_i) = \Xi_L[D_i]$.

\hfill\break
\textbf{Proof:} Notice that by choice of $G_i,\pi(G_i) = V(\Xi_L[D_i])$. Let $p,q\in\pi(G_i)$ and suppose $p-q\in\Xi_L[D_i]$. Without loss of generality suppose $p\in A, q\in B$. Then, there exist $j,k\in\mathbb{N}$ such that $q = q_j,p = p_k$. Let $(y_0,\cdots,y_{n_i})\in P_i$ have order $p_k$. Then, $(y_1,\cdots,y_{n_i}) = (1,\cdots,y_k,1,\cdots)$ for $y_k$ a generator of $C_{p_k}$. Let $(x_0,\cdots,x_{m_i})\in Q_i$ have order $q_j$. Then $(x_0,\cdots,x_{m_i}) = (1,\cdots, x_j,\cdots,1)$ for some nontrivial $x_j\in Q_{i,q_j}$. Then, $\varphi(1,\cdots,y_k,1,\cdots)(1,\cdots, x_j,\cdots,1) = (1,\cdots,1,\gamma_{i,q_j,k}(y_k)x_j,1,\cdots,1)$ is equal to $(1,\cdots,x_j,\cdots,1)$ if and only if $\gamma_{i,q_j,k}(y_k)x_j = x_j$. Because $x_j\neq 0,$ this would imply $\gamma_{i,q_j,k}(y_k) = 1$, which would be a contradiction to the fact that $y_k$ is nontrivial. Thus, $\varphi(1,\cdots,y_k,\cdots,1)$ does not fix $(1,\cdots,x_j,\cdots,1)$, so by \cite[Corollary 1.5]{Suz}, $p-q\in\pgc(G_i)$. 

Let $p,q\in\pi(G_i)$ and suppose $p-q\not\in\Xi_L[D_i]$. Without loss of generality suppose $p\in A, q\in B$. Then, there exist $j,k\in\mathbb{N}$ such that $q = q_j,p = p_k$. Let $y\in C_{p_k}$ and $x\in Q_{i,q_j}$ be nontrivial. Then, notice that $(1,\cdots,y,\cdots,1)$ has order $p_k$ in $P_i$ and $(1,\cdots,x,\cdots,1)$ has order $q_j$ in $Q_i$. Notice that $\varphi(1,\cdots,y,\cdots,1)(1,\cdots,x,\cdots,1) = (1,\cdots,\gamma_{i,q_j,k}(y)x,\cdots,1)$. Recall that, because $p-q\not\in\Xi_L[D_i]$, $\gamma_{i,q_j,k}(y) = 1$. Thus, $\varphi(1,\cdots,y,\cdots,1)$ fixes $(1,\cdots,x,\cdots,1)$, so $p-q\not\in\pgc(G_i)$ by \cite[Corollary 1.5]{Suz}.
\begin{flushright}
$\dashv_{\text{Claim}}$
\end{flushright}
Enumerate $\mathbb{P}\setminus V(\Xi_L)$ by $p_0,p_1,\cdots$. For each $i\in\mathbb{N}$, let $\varphi_i$ be the formula encoding $\Xi_L[D_i]$ as an induced subgraph of the prime graph complement of any group it is satisfied in, and let $\psi_i$ be $\neg\exists x(x^{p_i} = 1\land (\bigwedge_{1\leq j<p_i}x^j\neq 1))$. Let $\Phi'$ be $\Phi\cup\{\varphi_i:i\in\mathbb{N}\}\cup\{\psi_i:i\in\mathbb{N}\}$. Given any finite subset $\Delta$ of $\Phi'$, let $k$ be the maximum $k$ such that $\varphi_k\in\Phi'$. Then, by claims 1 and 2, $G_k\vDash\Delta$, so by compactness, $\Phi'$ is satisfiable. Let $G\vDash\Phi'$. Then, $G\vDash\Phi$, so $G$ is a locally torsion metabelian (and thus metanilpotent) group, and $G\vDash\varphi_i,G\vDash\psi_i$ for each $i\in\mathbb{N}$, so $\pgc(G) = \Xi_L$.  
\end{proof}
\begin{corollary}
    Let $\Xi$ be a countable graph. $\Xi$ is isomorphic to the prime graph complement of a locally torsion metabelian group if and only if $\Xi$ is 2-colorable.
\end{corollary}
\begin{proof}
    This follows from the proof of \thref{metabelian}.
\end{proof}
\section{Construction of $T$-Solvable Groups}\label{tsolv}
We first introduce some definitions related to elementary classes:
\begin{definition}
    Let $\mathcal{L}$ be a language. A class $\mathcal{C}$ of $\mathcal{L}$-structures is called \textbf{elementary} if there exists a set  of first-order $\mathcal{L}$-sentences $\Phi$ such that a $\mathcal{L}$-structure $\mathcal{M}\vDash\Phi$ if and only if $\mathcal{M}$ is a member of $\mathcal{C}$.
\end{definition}
\begin{definition}
    Let $\mathcal{L}$ be a language. A class $\mathcal{C}$ of $\mathcal{L}$-structures is called \textbf{basic elementary} if there exists a single $\mathcal{L}$-sentence $\phi$ such that a $\mathcal{L}$-structure $\mathcal{M}\vDash\phi$ if and only if $\mathcal{M}$ is a member of $\mathcal{C}$.
\end{definition}
\begin{definition}
    We define $\mathcal{C}_{\mathcal{T},k}$ to be the class of $\mathcal{T}$-solvable groups $G$ such that there exists a series $1=N_0\unlhd\cdots\unlhd N_k = G$ and each factor is abelian or isomorphic to an element of $\mathcal{T}$.
\end{definition}
\begin{proposition}\thlabel{axiomatizable}
    Let $\mathcal{T}$ be a finite collection of finite nonabelian simple groups. Let $k\in\mathbb{N}^+$. $\mathcal{C}_{\mathcal{T},k}$ is a basic elementary class in any expansion of the language $\mathcal{L} = \{1,^{-1},\cdot, N_i:0< i<k\}$ where $N_i,1\leq i\leq k-1$, is a unary predicate symbol.
\end{proposition}
\begin{proof}
    Notice that the quotient groups $N_{i+1}^G/N_i^G$ are interpretable. Because $T$ is finite, there exists a sentence $\varphi$ such that $G\vDash\varphi$ if and only if $G\cong T$. Let $A$ be the set of all strings over $\{0,1\}$ of length $k$, and for each $\sigma\in A$, let $\psi_\sigma$ be the $\mathcal{L}$-sentence that says ``$N_{i+1}/N_i$ is abelian if $\sigma(i) = 0$ and isomorphic to $T$ if $\sigma(i) = 1$".  It follows that $\bigvee_{\sigma\in A}\psi_\sigma$ axiomatizes the class.
\end{proof}
\begin{definition}
    Let $T$ be a nonabelian finite simple group. Let $G$ be a $T$-solvable group. A standard subnormal series of $G$ is a finite subnormal series such that each factor is either abelian or isomorphic to $T$. We define $n(G)$ to be be the minimum length of a standard subnormal series of $G$.
\end{definition}
\begin{definition}
    Let $\{\Lambda_i\}_{i\in\mathbb{N}}$ be a sequence of graphs. We say the sequence is nested if for each $i\in\mathbb{N}, \Lambda_i$ is an induced subgraph of $\Lambda_{i+1}$.
\end{definition}
\begin{lemma}\thlabel{handyUltrafilter}
    Let $N\in\mathbb{N}$. Let $T$ be a nonabelian finite simple group and let $\Xi_L$ be a labeled graph such that the following exist:
    \begin{enumerate}
        \item A sequence $\{\Lambda_i\}_{i\in\mathbb{N}}$ of finite nested induced subgraphs of $\Xi_L$.
        \item A sequence $\{G_i\}_{i\in\mathbb{N}}$ of finite $T$-solvable groups such that $\pgc(G_i) = \Lambda_i$ and $n(G_i)\leq N$ for all $i\in\mathbb{N}$.
        \item For each finite $\pi\subseteq V(\Xi_L)$, a $K_{\pi}\in\mathbb{N}^+$.
        \item For each finite $\pi\subseteq V(\Xi_L)$ and each $i\in\mathbb{N}$, a subgroup $H_{\pi,i}\leq G_i$ with exponent that divides $K_\pi$ such that $\pgc(H_{\pi,i}) = \pgc(G_i)[\pi(H_{\pi,i})]$.
    \end{enumerate}
    Then, there exists a locally torsion $T$-solvable group $G$ such that $\pgc(G) = \Xi_L$.
\end{lemma}
\begin{proof}
    Let $\mathcal{L}$ be an expansion of $\mathcal{L}_{gp}$ that contains unary predicates $N_i$ for $0\leq i\leq N$ and $H_\pi$ for every finite $\pi\subseteq V(\Xi_L)$. Let $\sigma$ be the sentence that axiomatizes the $\mathcal{C}_{T,N}$, guaranteed by \thref{axiomatizable}. For each finite $\pi\subseteq V(\Xi_L)$, let $\varphi_\pi$ be the sentence that says $H_\pi$ is a subgroup, for each $p\in \pi$ there exists an element of order $p$ in $H_\pi$, for all $x\in H_\pi$ $x^{K_\pi} = 1$, and encodes $\Xi_L[\pi]$ as the prime graph complement of $H_\pi$. For each $i\in\mathbb{N}$, let $\varphi_{\Lambda_i}$ encode $\Lambda_i$. 

    Enumerate the primes of $\mathbb{P}\setminus V(\Xi_L)$ by $p_0,p_1,\cdots$. For each $i\in\mathbb{N}$, let $\psi_i$ be the sentence $\neg\exists x(x^{p_i}=1\land(\land_{1\leq j<p_i}x^j\neq 1))$. Define $\Phi$ to be the set consisting of $\sigma, \varphi_\pi$ for each finite $\pi\subseteq V(\Xi_L),\varphi_{\Lambda_i}$ for each $i\in\mathbb{N}$, and $\psi_i$ for each $i\in\mathbb{N}$. By conditions (1)-(4), $\Phi$ is finitely satisfiable, so by compactness, there exists a $G\vDash\Phi$. Notice immediately that $G$ is a $T$-solvable group with $n(G)\leq N$ by the fact that $G\vDash\sigma$. By the inclusion of the $\psi_i,\varphi_{\Lambda_i}$ for $i\in\mathbb{N}$, $\pgc(G) = \Xi_L$, and $G$ is a locally torsion group by the inclusion of the axioms $\varphi_\pi$ in $\Phi$ for each finite $\pi\subseteq V(\Xi_L)$.
\end{proof}
We now generalize \cite[Lemma 1.3]{psl}:
\begin{lemma}\thlabel{13pslGeneral}
    Let $T$ be a finite nonabelian simple group, $F$ a strictly $T$-solvable group, $\Xi$ a graph such that each vertex has finite degree, and $X$ a set of $|\pi(F)|$ vertices of $\Xi$. Suppose $\Xi\setminus X$ is triangle-free and has a 3-coloring $\{\mathcal{O},\mathcal{D},\mathcal{I}\}$ such that all vertices in $N(X)$ are all colored $\mathcal{I}$. Further suppose that there exists some $E = F\ltimes L$ for some solvable $L$ such that $\pi(L)\subseteq\pi(T)$ and:
    \begin{itemize}
        \item There exists a graph isomorphism $\varphi$ from $\Xi[X]$ to $\pgc(E)$.
        \item Given $v\in \Xi\setminus X$ there exists a complex irreducible representation $V$ of $F$ such that for all $x\in X, x-v\in\Xi$ if and only if elements of order $\varphi(x)$ of $F$ act without fixed points in $V$.
        \item $N[\pi(L)]\subseteq\pi(T)$ in all prime graph complements of finite strictly $T$-solvable groups.
    \end{itemize}
    Then there exists a locally torsion ultragroup $G$ realizing $\Xi$.
\end{lemma}
\begin{proof}
Let $\Lambda$ be the graph obtained from $\Xi$ by (possibly) adding edges among vertices in $X$ such that $\Lambda[X]\cong\pgc(F)$, in particular, such that $E(\Xi)\subseteq E(\Lambda)$. Without loss of generality, by \cite[Corollary 5.3]{infinitePGCs}, there exists a labeling $\Lambda_L$ of $\Lambda$ and a sequence of induced subgraphs $\Lambda_L[X] = \Lambda_0\subseteq\Lambda_1\subseteq\cdots$ such that the following criteria are met:
\begin{enumerate}
    \item $\bigcup_{i\in\mathbb{N}}\Lambda_i = \Lambda_L$.
    \item For all $i\in\mathbb{N}$, there exists a finite strictly $T$-solvable group $H_i$ such that $\pgc(H_i)=\Lambda_i$.
\end{enumerate}
Notice also that $\Lambda_i[X] = \pgc(F)$ for all $i\in\mathbb{N}$.

Let $\Xi_L$ be the labeling of $\Xi$ obtained by considering $\Xi$ as a subgraph of $\Lambda$ and applying the vertex labeling of $\Lambda_L$ to the vertices of $\Xi$. Fix $i\in\mathbb{N}$. By the proof of \cite[Lemma 1.3]{psl}, there exists a finite strictly $T$-solvable group $K_i$ such that $\pgc(K_i) = \Xi_L[V(\Lambda_i)]$ and $n(K_i)\leq n(F) + n(L) + 3$ (to see this last bound, see the proofs of  \cite[Lemma 1.3]{psl} and \cite[Lemma 0.13]{psl}). For each $i\in\mathbb{N}$, let $\Xi_i = \Xi_L[V(\Lambda_i)]$. 

For each $i\in\mathbb{N}$, consider the directed graph corresponding to $\pgc(K_i)$ given in \cite[Lemma 0.13]{psl}.

Let $\pi\subseteq V(\Xi_L)$ be finite and contain $\pi(T)$. Now choose the least $i\in\mathbb{N}$ such that $\pi\subseteq\pi(K_i)$. By \cite[Theorem 4.4]{Suz}, there exists a Hall $\pi$-subgroup $H_{i,\pi}\leq K_i$ such that $\pgc(H_{i,\pi}) = \pgc(K_i)[\pi]$. By the constructions given in \cite[Lemma 5.1]{infinitePGCs}, \cite[Lemma 1.3]{psl}, and \cite[Lemma 0.13]{psl}; for all $j\geq i$, $|H_{j,\pi}|\leq |F||L|r_1^{k_{1,\pi}}\cdots r_{n_\pi}^{k_{n_\pi,\pi}}\prod_{p\in\pi\setminus\{r_1,\cdots,r_{n_{\pi}}\}}p$ for some $r_1,\cdots, r_{n_\pi}\in\pi$ and $k_{1,\pi},\cdots,k_{n_\pi,\pi}\in\mathbb{N}^+$ which depend only on $\pi$. Then, there exists $N_\pi\in\mathbb{N},N_\pi\geq |F||L|r_1^{k_{1,\pi}}\cdots r_{n_\pi}^{k_{n_\pi,\pi}}\prod_{p\in\pi\setminus\{r_1,\cdots,r_{n_{\pi}}\}}p$ such that for each $j\geq i,$ the exponent of $H_{j,\pi}$ divides $N_\pi$.

By choice of the sequence $\{\Lambda_i\}_{i\in\mathbb{N}}$, for all $i\in\mathbb{N}$, $\Xi_i$ is an induced subgraph of $\Xi_{i+1}$ and $ \Xi_L[X]=\Xi_i[X]$. Additionally, $\bigcup_{i\in\mathbb{N}}\Xi_i = \Xi_L$. Thus, by \thref{handyUltrafilter}, there exists a $T$-solvable locally torsion ultragroup $G$ such that $\pgc(G) = \Xi_L$.
    
\end{proof}

\section{$T$-Solvable Locally Torsion Classifications}\label{tsolvconstruction}
We consider the groups classified in \cite[Theorem 2.7]{psl}. To do this, we first specify these groups. Let $T$ be a finite nonabelian simple group such that $\pi(T) = \pi(\operatorname{Aut}(T))$ with the following prime graph complement (as in \cite[Section 2]{psl}):
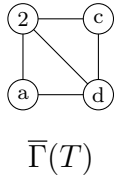
\begin{figure}[H]
    \centering
    \begin{minipage}{\textwidth}
        \centering
        \begin{tikzpicture}
            \coordinate (A) at (0, 1);   
            \coordinate (B) at (1, 1);   
            \coordinate (C) at (1, 0);   
            \coordinate (D) at (0, 0);   
            \draw (B) -- (C);
            \draw (C) -- (D);
            \draw (A) -- (D);
            \draw (A) -- (B);
            \draw (A) -- (C);
            \foreach \point in {A, B, C, D}
                \draw[draw, fill=white] (\point) circle [radius=0.2cm];
            \foreach \point/\name in {(A)/2, (B)/c, (C)/d, (D)/a}
                \node at \point {\scriptsize\name};
        \end{tikzpicture}
        \caption*{$\pgc(T)$}
    \end{minipage}%
\end{figure}
Also suppose that $T$ satisfies the following criteria as in \cite[Section 2]{psl}:
\begin{enumerate}
    \item The Schur multiplier of $T$ is 1.
    \item The fixed point information for $T$ matches \thref{psl16fpinfo} under some bijection between $\pi(T)$ and $\{2, a, c, d\}$ such that 2 is mapped to itself.
    \item The rooted graph \szunrecognizei (where the white vertex indicates the root) is realizable by $P\rtimes T$ for some $p$-group $P$ with $p\in\{2, a, c\}$.
    \item The graph \SuzukiThirtyTwoElusiveGraphWhiteBadb is not realizable by a $T$-solvable group.
    \item The Sylow 2-subgroups of $T$ do not satisfy the Frobenius criterion, \cite[Definition 2.2.1]{2023REU}.
\end{enumerate}
\begin{fact}\thlabel{psl16fpinfo}
    \begin{tabular}[h]{|c|c|}
    \hline
        Fixed Point Information &  \\
        \hline
        $T$ & $[[2, a, c], [2,a,c,d]]$ \\
        \hline
    \end{tabular}\\
\end{fact}
A list $[p_1,\cdots,p_n]$ corresponds to some complex irreducible character of $T$ and contains exactly the primes $p_1,\cdots,p_n$ such that for each $i\in[n]$, there exists some element of $T$ such that an element of order $a_i$ has fixed points in the representation. The list $[[p_{1,1},\cdots,p_{1,n_1}],\cdots,[p_{m,1},\cdots,p_{m,n_m}]]$ lists every such list corresponding to a complex irreducible character of $T$. The notation is similar to that in \cite{Suz} and \cite{psl}.
\begin{definition}
    A \textbf{rooted graph} is a graph $(v,V,E)$ for some $v\in V$. The vertex $v$ is called the \textbf{distinguished vertex}.
\end{definition}
\begin{theorem}\thlabel{psl216theorem}
    Let $\Xi$ be a countable graph such that each vertex has finite degree. $\Xi$ is isomorphic to the prime graph complement of some locally torsion $T$-solvable group with the finite Frobenius property if and only if $\Xi$ is 3-colorable and one of the following is true:
    \begin{enumerate}
        \item $\Xi$ is triangle-free.
        \item There exists a subset $X = \{w, x, y, z\}\subseteq V(\Xi)$ such that for all $p-q\in\Xi$ with $p\in X, q\in\Xi\setminus X$, we have that $p = z$. All triangles of $\Xi$ are contained in $X$. Additionally, in some 3-coloring of $\Xi$ the closed neighborhood $N(X)$ shares one color, and $\Xi[X] =\hamSandwichWhite$, \szunrecognizei, or \notIsolatedVertexWhite where $z$ is the white vertex.
    \end{enumerate}
\end{theorem}
\begin{proof}
    We first prove the forwards direction. Let $G$ be a locally torsion $T$-solvable group with the finite Frobenius property. By \cite[Lemma 2.8]{infinitePGCs} and \cite[Theorem 2.7]{psl}, $\pgc(G)$ is 3-colorable. If $\Xi$ is triangle-free, (1) is satisfied, so we may assume $\Xi$ contains a triangle. By \cite[Theorem 3.11]{infinitePGCs}, this implies $G$ is strictly $T$-solvable. Thus, by \cite[Theorem 2.4]{psl} and \cite[Lemma 4.1]{infinitePGCs}, for all edges $p-q\in\pgc(G)$ such that $p\in\pi(T),q\in\pi(G)\setminus\pi(T)$, $p = d$. By \cite[Theorem 2.4]{psl} and \cite[Lemma 4.3]{infinitePGCs}, there exists a 3-coloring of $\pgc(G)$ where $N(\pi(T))$ shares one color. We may apply \cite[Lemma 4.2]{infinitePGCs} to \cite[Theorem 2.4]{psl} to see that all triangles of $\pgc(G)$ are contained in $\pgc(G)[\pi(T)]$. Notice that $\pgc(G)[\pi(T)]\subseteq\pgc(T)$, so this fact and by criterion 4, $\pgc(G)[\pi(T)]$ cannot be isomorphic to \psltwosixteenimpossiblei, \psltwosixteenimpossibleii, \SuzukiThirtyTwoElusiveGraphWhiteBadb, or \psltwosixteenimpossibleiii. The only remaining graphs to which $\pgc(G)[\pi(T)]$ could be isomorphic are listed in (2), and thus, (2) is satisfied.

    We now turn to the backwards direction. If (1) is satisfied, then by \cite[Theorem 3.11]{infinitePGCs}, there exists a solvable torsion group with the finite Frobenius property realizing $\Xi$. If (2) is satisfied, by \cite[Fact 2.6]{psl} and \thref{psl16fpinfo}, there exists a strictly $T$-solvable finite group $F$ with an appropriate $L$ such that $E = F\ltimes L$ realizes $\Xi[X]$ such that $F$ and $L$ satisfy all of the appropriate conditions given in \thref{13pslGeneral}. Thus, by \thref{13pslGeneral}, there exists a locally torsion $T$-solvable group with the finite Frobenius property realizing $\Xi$.
\end{proof}
\begin{corollary}
    A countable graph $\Xi$ such that each vertex has finite degree is realized as the prime graph complement of a locally torsion $\operatorname{PSL}(2, 2^4)$-solvable group if and only if the conditions in \thref{psl216theorem} are satisfied.
\end{corollary}
\begin{proof}
    This follows from \cite[Subsection 2.2]{psl} and \thref{psl216theorem}.
\end{proof}

We now extend the classification of the prime graphs of $T$-solvable groups where $T$ is as in \cite[Section 5]{psl}. Namely, we suppose $T$ is a nonabelian finite simple group with the following associated prime graphs, from \cite{psl}:
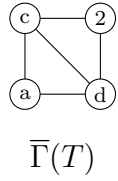
\begin{figure}[H]
    \centering
    \begin{minipage}{\textwidth}
        \centering
        \begin{tikzpicture}
            \coordinate (A) at (0, 1);   
            \coordinate (B) at (1, 1);   
            \coordinate (C) at (1, 0);   
            \coordinate (D) at (0, 0);   
            \draw (B) -- (C);
            \draw (C) -- (D);
            \draw (A) -- (D);
            \draw (A) -- (B);
            \draw (A) -- (C);
            \foreach \point in {A, B, C, D}
                \draw[draw, fill=white] (\point) circle [radius=0.2cm];
            \foreach \point/\name in {(A)/c, (B)/2, (C)/d, (D)/a}
                \node at \point {\scriptsize\name};
        \end{tikzpicture}
        \caption*{$\pgc(T)$}
    \end{minipage}%
\end{figure}
such that $\pi(T) = \pi(\operatorname{Aut}(T))$ and such that $T$ satisfies the following conditions:
\begin{enumerate}
    \item The Schur multiplier of $T$ is 2.
    \item The fixed point information for $T$ matches \thref{psl223fpinfogeneral}.
    \item The graph \szunrecognizei is realized by $P\rtimes T$ where $d$ is the white vertex and $P$ is a $p$-group, $p\in\{a, c\}$.
    \item The graph \SuzukiThirtyTwoElusiveGraphWhiteBadb where $d$ is the white vertex is not realizable by a $T$-solvable group.
\end{enumerate}
\begin{fact}\thlabel{psl223fpinfogeneral}
    \begin{tabular}[h]{|c|c|}
    \hline
        Fixed Point Information &  \\
        \hline
        $T$ & $[[2, a, c], [2, a,  c, d]]$ \\
        \hline 
        $2.T$ & $[[2, a, c], [2, a, c, d], [a, c], [a, c, d]]$ \\
        \hline
    \end{tabular}
\end{fact}
A list $[p_1,\cdots,p_n]$ corresponds to some complex irreducible character of $T$ and contains exactly the primes $p_1,\cdots,p_n$ such that for each $i\in[n]$, there exists some element of $T$ such that an element of order $a_i$ has fixed points in the representation. The list $[[p_{1,1},\cdots,p_{1,n_1}],\cdots,[p_{m,1},\cdots,p_{m,n_m}]]$ lists every such list corresponding to a complex irreducible character of $T$. The notation is similar to that in \cite{Suz} and \cite{psl}.
\begin{theorem}\thlabel{psl223thm}
    A countable graph $\Xi$ such that each vertex has finite degree is isomorphic to the prime graph complement of some locally torsion $T$-solvable group with the finite Frobenius property if and only if one of the following is true:
    \begin{itemize}
        \item[1.] $\Xi$ is triangle-free and 3-colorable.
        \item[2.] $\Xi$ contains a subset $X = \{w,x,y,z\}\subseteq V(\Xi)$ with a white vertex $z$ such that $N(X)$ share one color in some 3-coloring of $\Xi$ and all triangles of $\Xi$ are contained in $\Xi[X]$. In addition, one of the following hold:
        \begin{itemize}
            \item[a.] All edges in $\Xi$ between $p\in X$ and $q\in\Xi\setminus X$ have $p = z$ and $\Xi[X]\cong  \hamSandwichWhite$, \szunrecognizei, or \notIsolatedVertexWhite.
            \item[b.] All edges in $\Xi$ between $p\in X$ and $q\in\Xi\setminus X$ have $p = z$ or $p = x$ and $\Xi[X]\cong $\notIsolatedVertexWhiteGrey, where $x$ is the isolated vertex.
        \end{itemize}
    \end{itemize}
\end{theorem}
\begin{proof}
    First consider the forward direction, take $G$ to be a $T$-solvable group. By \cite[Lemma 4.6]{psl} and \cite[Lemma 2.8]{infinitePGCs}, $\pgc(G)$ is 3-colorable. If $\pgc(G)$ is triangle-free, (1) is satisfied, so we may assume $\pgc(G)$ contains a triangle. Then, by \cite[Theorem 3.11]{infinitePGCs}, $G$ is strictly $T$-solvable. By \cite[Lemma 4.7]{psl} and \cite[Lemma 4.2]{infinitePGCs}, the triangle is contained in $\pgc(G)[\pi(T)]$. Assign $X$ to $\pi(T)$ by assigning $w=a, x =2,y=c,z = d$. By \cite[Fact 4.18]{psl} and \cite[Lemma 4.1]{infinitePGCs}, $c - p, a-p\not\in\pgc(G)$ for any $p\in\pi(G)\setminus\pi(T)$. By \cite[Lemma 4.6]{psl} and \cite[Lemma 4.3]{infinitePGCs}, there exists a 3-coloring of $\pgc(G)$ such that $N(X)$ shares one color. By \cite[Lemma 2.6]{infinitePGCs} and \cite[Lemma 4.8]{psl}, there exists a finite strictly $T$-solvable group $H$ with $\pgc(G)[\pi(T)] = \pgc(H) \cong \hamSandwichWhite$, \szunrecognizei, or \notIsolatedVertexWhite. If $\pgc(G)[\pi(T)]\cong \hamSandwichWhite$ or \szunrecognizei, for any $p-q\in\pgc(G)$ with $p\in\pi(T),q\in\pi(G)\setminus\pi(T)$, there exists a finite $T$-solvable group $J$ with $\pgc(G)[\pi(T)\cup\{q\}] =\pgc(J)$ by \cite[Lemma 2.7]{infinitePGCs}, so $p = d$, or else we would reach a contradiction to \cite[Lemma 1.1]{psl}. In this case (2a) is satisfied.

    Now consider the case where $\pgc(G)[\pi(T)]\cong \notIsolatedVertexWhite$. If $2$ has no neighbors, (2a) is satisfied, otherwise, (2b) is satisfied.

    We now turn to the backwards direction. Suppose $\Xi$ satisfies (1). By \cite[Theorem 3.11]{infinitePGCs}, there exists a locally torsion solvable group with the finite Frobenius property $G$ realizing $\Xi$, so assume $\Xi$ satisfies one of the conditions in (2). If $\Xi$ satisfies (2a) and $\Xi[X]\cong \hamSandwichWhite$, there exists a locally torsion $T$-solvable group with the finite Frobenius property $G$ whose prime graph complement realizes $\Xi$ by \thref{psl223fpinfogeneral} and \cite[Lemma 5.1]{infinitePGCs}. If $\Xi[X]\cong \notIsolatedVertexWhite$, let $F = C_c\times T$, and notice that by using \thref{psl223fpinfogeneral} and \thref{13pslGeneral} we are guaranteed a locally torsion $T$-solvable group with the finite Frobenius property that realizes $\Xi$. Now suppose $\Xi[X]\cong \szunrecognizei$. By criterion 4, there exists a strictly $T$-solvable group $M$ realizing $\szunrecognizei$ where d
 is the white vertex. By \thref{psl223fpinfogeneral}, we may apply \thref{13pslGeneral} to find a locally torsion $T$-solvable group with the finite Frobenius property $G$ that realizes $\Xi$. 
 
 Finally, suppose $\Xi$ satisfies (2b). Let $F = 2.T$, and apply \cite[Lemma 5.1]{infinitePGCs} to \thref{psl223fpinfogeneral} to get a locally torsion $T$-solvable group with the finite Frobenius property $G$ realizing $\Xi$.
 \end{proof}

\begin{corollary}
Let $T$ be one of $\operatorname{PSL}(2, 11), \operatorname{PSL}(2,19)$ or $\operatorname{PSL}(2, 23)$. The prime graph complements of torsion $T$-solvable groups with the finite Frobenius property have been completely classified.     
\end{corollary}
\begin{proof}
    See \thref{psl223thm} and \cite[Section 4]{psl}.
\end{proof}
We now turn our attention to locally torsion $\operatorname{PSL}(2, 13)$-solvable groups with the finite Frobenius property. To begin, we include the fixed point information for certain $\operatorname{PSL}(2, 13)$-solvable groups, from \cite[Appendix A]{alland}:
\begin{fact}\thlabel{psl213fpinfogeneral}
    \begin{tabular}[h]{|c|c|}
    \hline
        Fixed Point Information &  \\
        \hline
        $\operatorname{PSL}(2, 13)$ & $[[2, 3, 7], [2, 3, 7, 13]]$ \\
        \hline 
        $\operatorname{Aut}(\operatorname{PSL}(2, 13))$ & $[[2,3,7], [2,3,7,13]]$ \\
        \hline
        $2.\operatorname{PSL}(2, 13)$ & $[[2,3,7], [2,3,7,13],[3],[3,7],[3,7,13]]$ \\
        \hline
    \end{tabular}
\end{fact}
A list $[p_1,\cdots,p_n]$ corresponds to some complex irreducible character of $T$ and contains exactly the primes $p_1,\cdots,p_n$ such that for each $i\in[n]$, there exists some element of $T$ such that an element of order $a_i$ has fixed points in the representation. The list $[[p_{1,1},\cdots,p_{1,n_1}],\cdots,[p_{m,1},\cdots,p_{m,n_m}]]$ lists every such list corresponding to a complex irreducible character of $T$.

The following lemma generalizes \cite[Lemma 2.2.1]{alland}:
\begin{lemma}\thlabel{221Analog}
    Let $G$ be a locally torsion strictly $\operatorname{PSL}(2, 13)$-solvable group. $3-p\not\in\pgc(G)$ for all $p\in\pi(G)\setminus\pi(\operatorname{PSL}(2, 13))$. Additionally, if $2-r\in\pgc(G)$ for any $r\in\pi(\operatorname{PSL}(2, 13))$, then $2-p\not\in\pgc(G)$ for all $p\in\pi(G)\setminus\pi(\operatorname{PSL}(2, 13))$.
\end{lemma}
\begin{proof}
    Let $p\in\pi(G)\setminus\pi(\operatorname{PSL}(2, 13))$. By \cite[Lemma 2.6]{infinitePGCs}, there exists a finite strictly $\operatorname{PSL}(2, 13)$-solvable group $H\leq G$ such that $\pgc(H) = \pgc(G)[\pi(\operatorname{PSL}(2, 13))\cup\{p\}]$. By \cite[Lemma 2.2.1]{alland}, $3-p\not\in\pgc(H) = \pgc(G)[\pi(\operatorname{PSL}(2, 13))\cup\{p\}]$. Thus, $3-p\not\in\pgc(G)$ for any $p\in\pi(G)\setminus\pi(\operatorname{PSL}(2, 13))$.

    Suppose $2-r\in\pgc(G)$ for any $r\in\pi(\operatorname{PSL}(2, 13))$, and let $p\in\pi(G)\setminus\pi(\operatorname{PSL}(2, 13))$. There exists a finite strictly $\operatorname{PSL}(2, 13)$-solvable group $H\leq G$ such that $\pgc(H) = \pgc(G)[\pi(\operatorname{PSL}(2, 13))\cup\{p\}]$ by \cite[Lemma 2.6]{infinitePGCs}. By \cite[Lemma 2.2.1]{alland}, $3-p\not\in\pgc(H) = \pgc(G)[\pi(\operatorname{PSL}(2, 13))\cup\{p\}]$, so $3-p\not\in\pgc(G)$.
\end{proof}

The following result generalizes \cite[Lemma 2.2.6]{alland}:
\begin{lemma}\thlabel{generalized226}
    Let $G$ be a locally torsion $\operatorname{PSL}(2, 13)$-solvable group $G$ satisfying the hypotheses of \cite[Theorem A]{alland}. Suppose $7-p\in\pgc(G)$ for some $p\in\pi(G)\setminus{\pi(\operatorname{PSL}}(2, 13))$. If $13-q\in\pgc(G)$ for some $q\in\pi(G)\setminus\pi(\operatorname{PSL}(2, 13))$, then $13-p\in\pgc(G)$. If $2-q\in\pgc(G)$ for some $q\in\pi(G)\setminus\pi(\operatorname{PSL}(2, 13))$, then $13-p\in\pgc(G)$.
\end{lemma}
\begin{proof}
    For contradiction, suppose the lemma fails. Then, there are two cases: The case where $13-p\not\in\pgc(G)$ but $7-p\in\pgc(G)$ and $13-q\in\pgc(G)$ for some $p,q\in\pi(G)\setminus\pi(\operatorname{PSL}(2, 13))$ and the case where $13-p\not\in\pgc(G)$ but $7-p\in\pgc(G)$ and $2-q\in\pgc(G)$ for some $p,q\in\pi(G)\setminus\pi(\operatorname{PSL}(2, 13))$. Let $A$ be a triangle contained in $\pgc(G)$. In the both cases, let $H\leq G$ be a finite subgroup realizing $\pgc(G)[\pi(\operatorname{PSL}(2, 13))\cup\{p,q\}\cup A]$, which exists by \cite[Lemma 2.6]{infinitePGCs}. By \cite[Proposition 2.2.4]{alland}, $\pgc(G)$ satisfies the hypotheses of \cite[Theorem A]{alland}. This is a contradiction to \cite[Lemma 2.2.6]{alland}.  
\end{proof}
The following lemma is the generalization of \cite[Lemma 2.2.3]{alland}.
\begin{lemma}\thlabel{generalized224}
    Let $G$ be a locally torsion $\operatorname{PSL}(2, 13)$-solvable group such that $\pgc(G)$ contains a triangle. Then, $13-7$ is an edge in that triangle.
\end{lemma}
\begin{proof}
    By \cite[Proposition 3.9]{infinitePGCs}, $G$ must be strictly $\operatorname{PSL}(2, 13)$-solvable. For contradiction, suppose $\pgc(G)$ contained a triangle $A$ that did not have the edge $13-7$. Let $H$ be a finite, strictly $\operatorname{PSL}(2, 13)$-solvable subgroup of $G$ such that $\pgc(G)[\pi(\operatorname{PSL}(2, 13))\cup A] = \pgc(H)$, guaranteed to exist by \cite[Lemma 2.6]{infinitePGCs}. By \cite[Proposition 2.2.4]{alland}, $H$ satisfies the hypotheses of \cite[Theorem A]{alland}. Thus, by \cite[Lemma 2.2.3]{alland}, $13-7$ is an edge in $\pgc(H)[A]$, a contradiction.
\end{proof}
The following lemma generalizes \cite[Lemma 2.2.5]{alland}:
\begin{lemma}\thlabel{225generalized}
    Let $G$ be a locally torsion $\operatorname{PSL}(2, 13)$-solvable group containing a triangle. If $2-r\in\pgc(G)$ for some $r\in\pi(\operatorname{PSL}(2, 13))$, then $7-p\not\in\pgc(G)$ for all $p\in\pi(G)\setminus\pi(\operatorname{PSL}(2, 13))$.
\end{lemma}
\begin{proof}
    Let $A$ be a triangle of $G$. By \cite[Proposition 3.9]{infinitePGCs}, $G$ is strictly $\operatorname{PSL}(2, 13)$-solvable. For contradiction, suppose $7-p\in\pgc(G)$ for some $p\in\pi(G)\setminus\pi(\operatorname{PSL}(2, 13))$. Let $H\leq G$ be a finite strictly $\operatorname{PSL}(2, 13)$-solvable group such that $\pgc(G)[\pi(\operatorname{PSL}(2, 13))\cup\{r, p\}\cup A] = \pgc(H)$, guaranteed to exist by \cite[Lemma 2.6]{infinitePGCs}. Then, by \cite[Proposition 2.2.4]{alland}, $H$ satisfies the hypotheses of \cite[Theorem A]{alland}, so by \cite[Lemma 2.2.5]{alland}, $7-p\not\in\pgc(H) = \pgc(G)[\pi(\operatorname{PSL}(2, 13))\cup\{r, p\}\cup A]$, a contradiction.
\end{proof}
The following result is similar to \cite[Theorem 2.2.8]{alland}:
\begin{theorem}
    An unlabeled countable graph $\Xi$ such that each vertex has finite degree is isomorphic to the prime graph complement of some locally torsion $\operatorname{PSL}(2, 13)$-solvable group if and only if one of the following holds:
    \begin{enumerate}
        \item $\Xi$ is triangle-free and 3-colorable.
        \item $\Xi$ contains a subset $X$ of four vertices such that in some 3-coloring of $\Xi$, $N(X)$ is colored with one color, and $X$ satisfies one of the following:
        \begin{itemize}
            \item[(a)] $X$ contains at most two triangle, all triangles of $\Xi$ are contained in $X$, three vertices of $X$ have no neighbors outside of $X$, and $\Xi[X]$ is connected. Moreover, if a vertex in $X$ indeed has neighbors outside $X$, then it is adjacent to all other vertices in $X$.
            \item[(b)] There exists a labeling of the vertices in $X$ by $\{a, b, c, d\}$ such that all triangles in $\Xi$ contain the $b-c$ edge, $a$ has no neighbors within $X$, $d$ has no neighbors outside of $X$, $N(c)\setminus\{b\}\subseteq N(b)$, and one of the following is satisfied:
            \begin{itemize}
                \item[(i)] $N(c)\setminus X \subseteq N(a)$.
                \item[(ii)] $N(a) = \varnothing$.
            \end{itemize}
        \end{itemize}
    \end{enumerate}
\end{theorem}
\begin{proof}
    Suppose $G$ is $\operatorname{PSL}(2, 13)$-solvable. By \cite[Lemma 2.2.2]{alland} and \cite[Lemma 2.8]{infinitePGCs}, $\pgc(G)$ is 3-colorable. If $\pgc(G)$ does not contain a triangle, (1) is satisfied, so assume $\pgc(G)$ does contain a triangle. Thus, $G$ is strictly $\operatorname{PSL}(2, 13)$-solvable by \cite[Proposition 3.9]{infinitePGCs}. By \cite[Lemma 2.2.1]{alland} and \cite[Lemma 4.1]{infinitePGCs}, $3-p\not\in\pgc(G)$ for all $p\in\pi(G)\setminus\pi(\operatorname{PSL}(2, 13))$. We consider cases based on the number of missing edges:

    \hfill\break
    \textbf{Case 1:} No edges of $\pgc(\operatorname{PSL}(2, 13))$ are missing. For contradiction, suppose there existed an edge $2-p$ or $7-p$ in $\pgc(G), p\in\pi(G)\setminus\pi(\operatorname{PSL}(2, 13))$. Then, by \cite[Lemma 2.6]{infinitePGCs}, there exists a finite strictly $\operatorname{PSL}(2, 13)$-solvable group $H$ with $\pgc(G)[\pi(\operatorname{PSL}(2, 13))\cup\{p\}]=\pgc(H)$, a contradiction to \cite[Lemma 1.1.3]{alland} and \cite[Lemma 2.2.5]{alland}, as $H$ satisfies the hypotheses of \cite[Theorem A]{alland}. Notice that only 13 can have edges to primes outside of $\pi(\operatorname{PSL}(2, 13))$, so given any subgraph $\Delta\subseteq\pgc(G)$ containing $\pi(\operatorname{PSL}(2, 13))$, all triangles of $\Delta$ are contained in $\pgc(G)[\pi(\operatorname{PSL}(2, 13))] = \Delta[\pi(\operatorname{PSL}(2, 13))]$ by \cite[Lemma 2.2.3]{alland}. Thus, (2a) is satisfied.

    \hfill\break
    \textbf{Case 2:} One edge of $\pgc(\operatorname{PSL}(2, 13))$ is missing. By \cite[Lemma 2.6]{infinitePGCs}, there exists a finite strictly $\operatorname{PSL}(2, 13)$-solvable subgroup of $G$, $H$, such that $\pgc(H) = \pgc(G)[\pi(\operatorname{PSL}(2, 13))]$. By \cite[Lemma 2.2.7]{alland} and \cite[Proposition 2.2.4]{alland}, $2-13, 3-13\in\pgc(G)$. Then, 2 is not isolated from the rest of $\pi(\operatorname{PSL}(2, 13))$, so by \thref{225generalized} and \thref{221Analog}, there exist no edges from 2 and 7 to $\pi(G)\setminus\pi(\operatorname{PSL}(2, 13))$. Thus, primes 2, 3, and 7 do not have edges to primes in $\pi(G)\setminus\pi(\operatorname{PSL}(2, 13))$. Then, the only prime in $\pi(\operatorname{PSL}(2, 13))$ which can have edges to primes outside of $\pi(\operatorname{PSL}(2, 13))$ is 13, so by \thref{generalized224}, the possible triangles are exactly one of $\{2, 7, 13\}$ or $\{3, 7, 13\}$, so $\pgc(G)$ satisfies (2a).

    \hfill\break
    \textbf{Case 3:} Two edges of $\pgc(\operatorname{PSL}(2, 13))$ are missing. We consider subcases on the edges $2-7$ and $2-13$.
    \hfill\break
    \textit{Case 3a:} $2-7, 2-13\in\pgc(G)$. By \thref{225generalized} and \thref{221Analog}, there are no edges from $\{2, 3, 7\}$ to $\pi(G)\setminus\pi(\operatorname{PSL}(2, 13))$. Notice that $3-7, 3-13$ must have been removed, so $\pgc(G)$ satisfies condition 2(b)(ii) with the labeling $3\mapsto a, 7\mapsto c, 13\mapsto b,$ and $2\mapsto d$.
\hfill\break
\textit{Case 3b: Exactly one of 2-7 or 2-13 is in $\pgc(G)$. } We will show that this case cannot occur by showing that it leads to a contradiction. By \thref{225generalized} and \thref{221Analog}, there are no edges from $\{2, 3, 7\}$ to $\pi(G)\setminus\pi(\operatorname{PSL}(2, 13))$. By \thref{generalized224}, all triangles must be contained within $\pgc(G)[\pi(\operatorname{PSL}(2, 13))]$. But then, because an edge has been removed from both triangles of $\pgc(\operatorname{PSL}(2, 13))$, $\pgc(G)$ is triangle-free, a contradiction.
\hfill\break
\textit{Case 3c: $2-7, 2-3\not\in\pgc(G)$}. Notice that $N(2)\cap\pi(\operatorname{PSL}(2, 13)) = \varnothing$. We first consider the case where there are no edges from $13$ to primes in $\pi(G)\setminus\pi(\operatorname{PSL}(2,13))$. Notice that $\pgc(G)$ satisfies 2(b)(i) with the labeling $2\mapsto a, 7\mapsto b, 13\mapsto c$, and $3\mapsto d$. Now consider the case where 2 is isolated from $\pgc(G)\setminus\pi(\operatorname{PSL}(2, 13))$ but 13 is not. By \thref{generalized226}, $\pgc(G)$ satisfies 2(b)(ii) with the labeling $2\mapsto a, 13\mapsto b, 7\mapsto c$, and $3\mapsto d$. Similarly, if $2$ and 13 both have edges to $\pi(G)\setminus\pi(\operatorname{PSL}(2, 13))$, by \thref{generalized226}, $\pgc(G)$ satisfies 2(b)(i) with the labeling $2\mapsto a, 13\mapsto b, 7\mapsto c,$ and $3\mapsto d$.

\hfill\break
\textbf{Case 4:} Three edges of $\pgc(\operatorname{PSL}(2, 13))$ are missing. Let $A$ be the triangle contained in $\pgc(G)$. Then, by \thref{generalized224}, $7-13\in A$. Thus, we consider possibilities based on which other edge remains. If one of $2-7, 2-13$ is in $\pgc(G)$, then by \thref{225generalized}, there exist no edges $7-p$ in $\pgc(G)$ with $p\in\pi(G)\setminus\pi(\operatorname{PSL}(2, 13))$. This contradicts the existence of a triangle in $\pgc(G)$ by \thref{generalized224}. By \cite[Lemma 2.6]{infinitePGCs}, there exists a finite strictly $\operatorname{PSL}(2, 13)$-solvable $H\leq G$ such that $\pgc(H) = \pgc(G)[\pi(\operatorname{PSL}(2, 13))]$. $3-7\not\in\pgc(H) = \pgc(G)[\pi(\operatorname{PSL}(2, 13))]$ by \cite[Lemma 2.2.7]{alland}. Thus, the only remaining case is the case where $3-13\in \pgc(G)$. Notice then that $N(2)\cap\pi(\operatorname{PSL}(2, 13))=\varnothing$, so for a triangle to exist, $13-p\in\pgc(G)$ for some $p\in\pi(G)\setminus\pi(\operatorname{PSL}(2, 13))$. Let $A$ be a triangle of $\pgc(G)$. By \cite[Lemma 2.6]{infinitePGCs}, there exists a finite strictly $\operatorname{PSL}(2, 13)$-solvable group $H\leq G$ such that $\pgc(H) = \pgc(G)[\pi(\operatorname{PSL}(2, 13))\cup A]$. $\pgc(H)$ contains a triangle, so by \cite[Proposition 2.2.4]{alland}, $H$ satisfies the hypothesis of \cite[Theorem A]{alland}. Thus, by \cite[Lemma 2.2.7]{alland}, $3-7\not\in\pgc(H) = \pgc(G)[\pi(\operatorname{PSL}(2, 13))\cup A]$. If $2-q\in\pgc(G)$ for some $q\in\pi(G)\setminus\pi(\operatorname{PSL}(2, 13))$ then $\pgc(G)$ satisfies 2(b)(i). If no such edge exists, by \thref{generalized226}, $\pgc(G)$ satisfies 2(b)(ii). In both cases, the labeling is given by $2\mapsto a, 13\mapsto b, 7\mapsto c, 3\mapsto d$.

\hfill\break
\textbf{Case 5:} Four edges of $\pgc(\operatorname{PSL}(2, 13))$ are missing. By \thref{generalized224}, the only remaining edge is $13-7$, so $N(2)\cap\pi(\operatorname{PSL}(2, 13)) = \varnothing$. Also, $13-p\in\pgc(G)$ for some $p\in\pi(G)\setminus\pi(\operatorname{PSL}(2, 13))$ to form a triangle. By \thref{generalized226}, if there does not exist an edge from 2 to a prime in $\pi(G)\setminus\pi(\operatorname{PSL}(2, 13))$, then $\pgc(G)$ satisfies 2(b)(ii), but if such an edge exists, then $\pgc(G)$ satisfies 2(b)(i) -- in both cases, with the labeling $2\mapsto a, 13\mapsto b, 7\mapsto c, $ and $3\mapsto d$.

By \thref{generalized224}, these cases are exhaustive.

\hfill\break
We now consider the backwards direction. If $\Xi$ satisfies (1), then by \cite[Theorem 3.11]{infinitePGCs}, $\Xi$ is realized by a torsion solvable group with the finite Frobenius property, so we may assume that $\Xi$ contains a triangle. We proceed by cases:

\hfill\break
\textbf{Case 1:} $\Xi$ matches condition 2(a). Notice that, depending on $\Xi[X]$, there exists a graph isomorphism from $E = \operatorname{PSL}(2, 13)$ to $\Xi[X]$ or $E = \operatorname{Aut(PSL}(2, 13))$ to $\Xi[X]$. If $N[X] = X$, then by \cite[Lemma 5.1]{infinitePGCs}, $\Xi$ is the prime graph complement of a $\operatorname{PSL}(2,13)$-solvable group. If there exists $v\in X$ such that $v-w\in\Xi$ for some $w\in\Xi\setminus X$, we may construct a graph isomorphism from $\Xi[X]$ to $\pgc(\operatorname{PSL}(2, 13))$ or $\pgc(\operatorname{Aut(PSL(}2, 13)))$ such that $\varphi(v) = 13$ by the adjacency condition in 2(a). Thus, letting $E $ be $\operatorname{PSL}(2, 13)$ or $\operatorname{Aut}(\operatorname{PSL}(2, 13))$ as appropriate, we may apply \cite[Lemma 5.1]{infinitePGCs} to \thref{psl213fpinfogeneral} to get a $\operatorname{PSL}(2, 13)$-solvable group $G$ realizing $\Xi$.

\hfill\break
\textbf{Case 2:} $\Xi$ matches condition 2(b)(i). Notice that, because $\Xi$ contains a triangle, there must be at least one edge in $\Xi[X]$. By the graph structures given in \cite[Appendix A]{alland}, there exists a graph isomorphism $\varphi:\Xi[X]\to\pgc(E)$ where $E$ is the appropriate one of $2.\operatorname{PSL}(2, 13), \mathbb{F}_{3}^{36}\rtimes2.\operatorname{PSL}(2, 13)$, or $C_3\times 2.\operatorname{PSL}(2, 13)$ such that $\varphi(a) = 2, \varphi(b)=13, \varphi(c) = 7$, and $\varphi(d) = 3$. By \thref{221Analog} and our assumptions on a possible coloring of $\Xi$, we can see that $\Xi$ satisfies the hypotheses of \cite[Lemma 5.1]{infinitePGCs} if $E = 2.\operatorname{PSL}(2, 13)$ or \thref{13pslGeneral} if $E = \mathbb{F}_{3}^{36}\rtimes2.\operatorname{PSL}(2, 13)$ or $E = C_3\times 2.\operatorname{PSL}(2, 13)$ by using the appropriate fixed point information in \thref{psl213fpinfogeneral}. The appropriate representations to use in each of these cases are described in full detail in the proof of \cite[Theorem 2.2.8]{alland}.

\hfill\break
\textbf{Case 3:} $\Xi$ matches condition 2(b)(ii). Build $\Xi'$ which satisfies 2(b)(i) by adding the edge $a-v$ to $\Xi$ for all $v\in N(c)\setminus X$. Let $G$ be a $\operatorname{PSL}(2, 13)$-solvable group realizing $\Xi'$ according to case 2(b)(i). Then, $G\times C_2$ realizes $\Xi$.
\end{proof}
\section{Reduction Results}\label{sec7}
\begin{proposition}\thlabel{infiniteCard}
   Let $\mathcal{L}$ be a countable language and $\Phi$ be a set of $\mathcal{L}$-sentences. Let $G$ be an infinite group, considered as a $\mathcal{L}$-structure $\mathcal{G}$, such that $\mathcal{G}\vDash\Phi$. For each infinite cardinal $\kappa$, there exists a group $H$ of order $\kappa$ such that $\mathcal{H}\vDash\Phi$ and $\pgc(G) = \pgc(H)$.
\end{proposition}
\begin{proof}
This follows from Lowenheim-Skolem and the fact that $\pgc(G)$ is encoded in the theory of $G$.

\end{proof}

\begin{corollary}
    Let $\mathcal{T}$ be a finite (possibly empty) collection of nonabelian finite simple group, and $G$ be a $\mathcal{T}$-solvable group that realizes a labeled graph $\Xi$. For any infinite cardinality $\kappa$, there is a $\mathcal{T}$-solvable group of cardinality $\kappa$ realizing $\Xi$.
\end{corollary}
\begin{proof}
    By \thref{infiniteCard}, it suffices to show that there exists a set $\Phi$ of $\mathcal{L}$-sentences that axiomatizes the class of groups $H$ with a subnormal series $1=N_0\unlhd\cdots\unlhd N_{n(G)} = H$ where each factor is either abelian or isomorphic to an element of $\mathcal{T}$. Let $\varphi$ be the sentence axiomatizing $\mathcal{C}_{\mathcal{T},n(G)}$, from \thref{axiomatizable}, and let 
    \begin{align*}
        \Phi = \{\varphi\}\cup\{\exists x\left(\left(\bigwedge_{i=1}^{r-1}x^i\neq 1\right)\land x^r = 1\right): r\in\mathbb{N}^+\text{ and there exists a }g\in G, o(g) = r\}\\
        \cup \{\forall x\left(x^r \neq 1\right): r\in\mathbb{N}^+\text{ for all }g\in G, o(g) \neq  r\}.
    \end{align*}
    Notice that $\Phi$ axiomatizes the class of $\mathcal{T}$-solvable groups with prime graph $\pgc(G)$. Thus, by \thref{infiniteCard}, the result follows.
\end{proof}
\begin{remark}
    Notice that in the above, the subclass of $\mathcal{C}_{\mathcal{T}, k}$ that realizes a specific labeled graph $\Xi$ is also axiomatizable, and thus elementary. However, depending on the graph, it may fail to  be basic elementary.
\end{remark}
We have proven that, given a graph that is realized as the prime graph of a group, the class of nonisomorphic groups that realize the graph is proper by \thref{infiniteCard}. This makes the question of how many nonisomorphic possibly infinite groups realize a given graph uninteresting, however, we may still ask about the number of nonisomorphic groups of a certain cardinality that realize a given graph.
\begin{proposition}
    Let $\Xi_L$ be a labeled graph that is realizable as the prime graph complement of a group $G$ such that $\omega(G)$ is infinite. Then, there exist at least three nonisomorphic but elementarily equivalent countable groups with prime graph complement $\Xi_L$.
\end{proposition}
\begin{proof}
    Let $G$ be a group that realizes $\Xi_L$. Let $X\subseteq G$ be a countable set. By Lowenheim-Skolem, there exists a $H_1\preceq G$ such that $X\subseteq H_1$ and $|H_1| = \aleph_0$. Then, because $H_1\equiv G$, $\pgc(G) = \pgc(H_1)$ by \thref{elmequiv}. 

    Let $T = \operatorname{Th}(H_1,\cdot, 1, ^{-1})$. By the fact that $\omega(G)$ is infinite, $H_1$ cannot be $\aleph_0$-categorical, so there exists a $H_2\not\cong H_1$ such that $H_1\equiv H_2$ and $|H_1|=|H_2|=\aleph_0$. By \cite[Theorem 4.4.6]{marker}, there exists a third $H_3\vDash T$ with $|H_3|=\aleph_0$ and $H_1\equiv H_2\equiv H_3\equiv G$ such that $H_3$ is not isomorphic to $H_1$ or $H_2$. Thus, by \thref{elmequiv}, $\pgc(H_1) = \pgc(H_2) = \pgc(H_3) = \pgc(G)$.
\end{proof}
\begin{corollary}
    Let $\Xi_L$ be an infinite labeled graph that is realizable as the prime graph complement of a group. Then, there exist at least three nonisomorphic but elementarily equivalent countable groups with prime graph complement $\Xi_L$.
\end{corollary}
\begin{proof}
    Let $H$ be a group that realizes $\Xi$. Let $X\subseteq H$ be a countable set. By the Lowenheim-Skolem theorem, there exists a $G\preceq H$ such that $X\subseteq G$ and $|G| = \aleph_0$. Then, because $G\equiv H$, $\pgc(G) = \pgc(H)$ by \thref{elmequiv}. 

    Let $T = \operatorname{Th}(G,\cdot, 1, ^{-1})$. Enumerate the element orders of $G$ by $n_1,n_2,\cdots$. Then, let $g_1,g_2,\cdots$ be a sequence of elements of $G$ such that $g_i$ has order $n_i$ for each $i\in\mathbb{N}^+$. By \cite[Theorem 4.4.1]{marker}, there exist at least two $\mathcal{L}_{gp}$-structures $H_1, H_2$ such that $H_1\vDash T, H_2\vDash T$ and $|H_1|,|H_2|\leq\aleph_0$. By \cite[Theorem 4.4.6]{marker}, there exists a third $H_3\vDash T$ with $|H_3|\leq\aleph_0$. Notice that $H_1\equiv H_2\equiv H_3\equiv G$, so by \thref{elmequiv}, $\pgc(H_1) = \pgc(H_2) = \pgc(H_3) = \pgc(G)$.
\end{proof}
\begin{theorem}
   Let $\Xi_L$ be an infinite labeled graph. Let $T$ be a consistent $\mathcal{L}_{gp}$-theory extending the theory of groups with quantifier elimination such that, for every model $\mathcal{G}\vDash T, \omega(G)$ is infinite. Then, there exists a torsion group $G$ with $\pgc(G) = \Xi_L$.
\end{theorem}
\begin{proof}
    Let $p(x)$ be the partial 1-type $\{x^n\neq 1:n\in\mathbb{N}^+\}$. By the omitting types theorem, it suffices to show that $p(x)$ is nonprincipal. For contradiction, suppose otherwise. Then, there exists a formula $\varphi(x)$ such that $T\vDash\forall x(\varphi(x)\to\psi(x))$ for each $\psi(x)\in p(x)$. By quantifier elimination, $\varphi(x)$ is logically equivalent to $\bigvee_{i=1}^n\bigwedge_{j_i=1}^{m_i}\varphi_{i,j}(x)$ for some $ n,m_i\in\mathbb{N}$ with $i\in[n]$ such that for each $i\in [n],j_i\in[m_i]$, $\varphi_{i,j_i}(x)$ is logically equivalent to one of the following:
    \begin{itemize}
        \item $x^{a_{i,j_i}} = 1$ for some $a_{i,j}\in\mathbb{Z}$.
        \item $x^{a_{i,j_i}}\neq 1$ for some $a_{i,j}\in\mathbb{Z}$.
    \end{itemize}
    Because $\bigvee_{i=1}^n\bigwedge_{j_i=1}^{m_i}\varphi_{i,j_i}(x)$ is equivalent to $\bigvee_{i=1,i\neq i_1,\cdots,i_l}^n\bigwedge_{j_i=1}^{m_i}\varphi_{i,j_i}(x)$ where $\bigwedge_{j_{i_k}=1}^{m_{i_k}}\varphi_{i_k,j_{i_k}}(x)$ is a contradiction in any model of $T$, we may assume that $\bigwedge_{j_i=1}^{m_i}\varphi_{i,j_i}(x)$ is not a contradiction for each $i\in[n]$. Then, $T\vDash\forall x(\bigwedge_{j_1=1}^{m_1}\varphi_{1,j_1}(x)\to \psi(x))$ for each $\psi(x)\in p(x)$. Notice that $\bigwedge_{j_1=1}^{m_1}\varphi_{1,j_1}(x)$ is logically equivalent to $\bigwedge_{j_1=1,j_1\neq {j_1}_1,\cdots, {j_1}_k}^{m_1}\varphi_{1,j_1}(x)$ where $\varphi_{1,{j_1}_i}(x)$ is a tautology for every $i\in[k]$, so we may assume each $\varphi_{1,j_1}(x)$ is not a tautology in every model of $T$. Thus, for each $j\in[m_1],\varphi_{1,j}(x)$ is logically equivalent to $x^{a_{1,j}}=1$ or $x^{a_{1,j}}\neq 1$ for some $a_{1,j}\in\mathbb{Z}\setminus\{0\}$.    
    If there exists a $j\in[m_1]$ such that $\varphi_{1,j}(x)$ is equivalent to $x^{a_{1,j}}=1$, then there exists some model $\mathcal{G}\vDash T$ such that $\mathcal{G}\vDash\exists x(x^{a_{1,j}}=1)$, and then $\mathcal{G}\vDash\exists x(x^{a_{1,j}}=1\to x^{a_{1,j}}\neq 1)$, a contradiction. Thus, $\bigwedge_{j=1}^{m_1}\varphi_{1,j_1}(x)$ is logically equivalent to $\bigwedge_{j=1}^{m_1}x^{a_{1,j}}\neq 1$ for $a_{1,j}\in\mathbb{Z}\setminus\{0\}$ for each $j\in[m_1]$. Then, let $\mathcal{G}\vDash T$. There exists a $g\in\mathcal{G}$ such that $o(g)>\max_{j\in[m_1]}|a_{1,j}|$, so $\mathcal{G}\vDash\bigwedge_{j=1}^{m_1}\varphi_{1,j}(g)\to g^{o(g)}\neq 1$, a contradiction. 
\end{proof}
\begin{corollary}
   Let $G$ be a group with infinite prime graph such that $\operatorname{Th}(G,\cdot,1,^{-1})$ has quantifier elimination. Then, there exists a torsion group $H$ with $\pgc(H) = \pgc(G)$.
\end{corollary}
\section{Outlook}\label{outlook}
There are many directions for future work, particularly for the reader interested in model theory. One could examine the prime graphs of groups of finite Morley rank, or look at when a graph is realizable by a group of finite Morley rank. One could also classify the prime graphs realizable by groups with $\omega$-stable theories. Another avenue for future work would be to completely classify the labeled graphs that are realizable by torsion groups.

One could classify the graphs that are realized as the prime graphs of locally finite groups. This particular direction of future work could help resolve \cite[19.52]{maslovanotebook}, an open question as to whether there exists a triangle-free but not three-colorable finite graph that is realized as the prime graph complement of a finite group. 

An approach that may be helpful for several of these directions would be to determine when a fragment of a consistent theory satisfiable in a group is satisfiable in a finite model, a model that omits all types corresponding to elements of infinite order, or satisfiable in a locally finite group.
\section*{Acknowledgments}
Thanks to Sylvia Carlisle, Thomas Keller, and Anand Pillay for their helpful comments on this draft.

\end{document}